\documentclass[12pt]{amsart}
\usepackage{amssymb,amsmath,amsthm,tipa}
\usepackage{tikz}

\usetikzlibrary{arrows,shapes,automata,backgrounds,decorations,petri,positioning,topaths}

\theoremstyle{plain}
\newtheorem{thm}{Theorem}[section]
\newtheorem{lem}[thm]{Lemma}

\newtheorem{cor}[thm]{Corollary}
\newtheorem{obs}[thm]{Observation}
\newtheorem{prop}[thm]{Proposition}
\newtheorem{ques}[thm]{Question}
\theoremstyle{definition}
\newtheorem{defn}[thm]{Definition}
\newtheorem{remark}[thm]{Remark}
\newtheorem{ex}[thm]{Example}

\newcommand{\bm}[1]{\mbox{\boldmath $#1$}}
\newcommand{\mf}[1]{\mbox{$\mathfrak #1$}}

\DeclareMathOperator{\Dy}{Dyck}
\DeclareMathOperator{\lr}{lr-max}
\DeclareMathOperator{\inv}{inv}
\DeclareMathOperator{\supp}{supp}
\DeclareMathOperator{\drop}{drops}
\DeclareMathOperator{\des}{des}
\DeclareMathOperator{\exc}{exc}
\DeclareMathOperator{\Cat}{Cat}

\DeclareMathOperator{\dep}{dp}

\title{The depth of a permutation}

\author[T. K. Petersen]{T.~Kyle Petersen}
\email{tkpeters@math.depaul.edu}
\author[B. E. Tenner]{Bridget Eileen Tenner$^{\dagger}$}
\email{bridget@math.depaul.edu}
\address{Department of Mathematical Sciences, DePaul University, Chicago, IL 60614}

\thanks{$^{\dagger}$ Research partially supported by a Simons Foundation Collaboration Grant for Mathematicians.}


\begin{document}

\begin{abstract}
For the elements of a Coxeter group, we present a statistic called \emph{depth}, defined in terms of factorizations of the elements into products of reflections. Depth is bounded above by length and below by the average of length and reflection length.  In this article, we focus on the case of the symmetric group, where we show that depth is equal to $\sum_i \max\{w(i)-i, 0\}$. We characterize those permutations for which depth equals length: these are the 321-avoiding permutations (and hence are enumerated by the Catalan numbers). We also characterize those permutations for which depth equals reflection length: these are permutations avoiding both 321 and 3412 (also known as boolean permutations, which we can hence also enumerate). In this case, it also happens that length equals reflection length, leading to a new perspective on a result of Edelman.\\

\end{abstract}

\maketitle

\section{Introduction}

There is a simple recursive sorting algorithm, called ``straight selection sort" by Knuth \cite{K}, that finds the largest misplaced value in a list and moves it to its correct position using a single transposition. The procedure is then repeated until the list is entirely sorted.  For example, $2431756$ would be sorted as follows, where the largest misplaced value at each stage is in boldface, and each transposition has been labeled by the pair of positions that it swaps.
\[ 2431\bm{7}56 \xrightarrow{(57)} 2431\bm{6}57 \xrightarrow{(56)} 2\bm{4}31567 \xrightarrow{(24)} \bm{2}134567 \xrightarrow{(12)} 1234567. \]
When the initial list is viewed as the one-line notation of a permutation, this algorithm yields a (not necessarily unique) minimal length expression for the permutation as a product of transpositions:
\[ w = 2431756 = (12)(24)(56)(57).\]
We compute the ``cost" of this algorithm in terms of the distances between the positions transposed at each step.  So in the example above, this would be \[ (2-1) + (4-2) + (6-5) + (7-5) = 6.\] In \cite{petersen 1}, this statistic was called the \emph{sorting index}. The sorting index was shown to have the same distribution over all permutations of $n$ as the inversion number; that is, it is a \emph{Mahonian} statistic (see \cite[Corollary 3]{petersen 1} and \cite{wilson}). 

While straight selection sort optimizes the number of transpositions needed to sort a list --- equivalently, to turn a permutation into the identity --- it does not necessarily optimize cost.  For example, there is an alternative route from our example $w=2431756$ to the identity:
$$2431\bm{7}56 \xrightarrow{(56)} 24315\bm{7}6 \xrightarrow{(67)} 2\bm{4}31567 \xrightarrow{(24)} \bm{2}134567 \xrightarrow{(12)} 1234567.$$
This sorting of the permutation yields
$$w=(12)(24)(67)(56),$$
and
$$(2-1) + (4-2) + (7-6) + (6-5) = 5 < 6.$$
In fact, $5$ is the lowest possible cost for this particular $w$.

The purpose of this paper is to characterize this kind of minimum sorting measure, which we call the \emph{depth} of a permutation, denoted $\dep(w)$. The term ``depth" is natural from the perspective of Coxeter groups, as we explain in Section \ref{sec:depth}. 

We will relate depth to the well understood notions of \emph{length} and \emph{reflection length}, denoted $\ell_S(w)$ and $\ell_T(w)$, respectively ($S$ refers to the set of simple reflections in the Coxeter group, $T$ refers to the set of all reflections).  For permutations, it is well known that the length is the same as the inversion number, and the reflection length is the size of the permutation minus the number of cycles.  We will show that depth is bounded above by length and below by the average of length and reflection length. 

Beyond preliminaries, we study depth primarily in the case of the symmetric group. In particular, we have the following characterization of depth.

\begin{thm}\label{thm:char}
Let $w \in \mf{S}_n$. Then its depth is given by \[ \dep(w) = \sum_{w(i) > i} (w(i)-i).\]
\end{thm}

We prove Theorem \ref{thm:char} by exhibiting an algorithm, based on straight selection sort, that achieves a factorization of $w$ having both minimum depth and minimum reflection length. 

\begin{remark}(Total Displacement)\label{rem:TD}
Diaconis and Graham \cite{DG} compared the statistics of length, reflection length, and what they called ``Spearman's disarray" for permutations. The disarray statistic, which Knuth calls ``total displacement" \cite[Problem 5.1.1.28]{K}, is defined as \[ \sum_{i=1}^n |w(i)-i|=2\sum_{w(i) > i} (w(i)-i).\] Thus according to Theorem \ref{thm:char}, the displacement of a permutation is exactly twice its depth. Several of the results proved here were first discovered by Diaconis and Graham in the context of total displacement, and we will highlight these connections as they arise. A recent paper of Guay-Paquet and the first author computes the generating function for total displacement \cite{GP}.
\end{remark}

A position $i$ such that $w(i)> i$ is called an \emph{excedance} of $w$, so the measure in Theorem~\ref{thm:char} is the sum of the sizes of the excedances. It is well known that excedance number and \emph{descent number} are equidistributed over $\mf{S}_n$. (Their common generating function is the Eulerian polynomial.) In a similar vein, we find that summing the sizes of descents gives a statistic with the same distribution as depth. That is, define the \emph{descent drop} of $w$ to be \[ \drop(w) = \sum_{w(i)>w(i+1)} ( w(i) - w(i+1)).\]
(Chung, Claesson, Dukes, and Graham study descent numbers in relation to ``maximum drop size,'' that is, $\max\{ j-w(j) \}$ in \cite{CCDG}. Apart from similar terminology, we know of no deeper connection between the current work and theirs.)

It is generally not true that $\dep(w) = \drop(w)$. For example, $\dep(3241) = 2+1=3$ while $\drop(3241) = 1+3=4$. Nonetheless, we will prove the following via a bijection of Steingr\'imsson \cite[Appendix]{stein}.

\begin{thm}\label{thm:equidist}
For all $n$, the pairs of statistics $(\des,\drop)$ and $(\exc,\dep)$ are equidistributed over $\mf{S}_n$. That is, \[ \sum_{w \in \mathfrak{S}_n} q^{\drop(w)} t^{\des(w)} = \sum_{w \in \mathfrak{S}_n} q^{\tiny \dep(w)} t^{\exc(w)}.\] In particular, \[ |\{ w \in \mf{S}_n : \dep(w) =k \}| = |\{ w \in \mf{S}_n : \drop(w) = k\}|.\]
\end{thm}

We are also able to characterize, in terms of pattern avoidance, those permutations for which the depth equals length and those for which depth equals reflection length. Diaconis and Graham include these results without proof as Remarks 3 and 4 of \cite{DG}. 

\begin{thm}\label{thm:extremes}
Let $w \in \mf{S}_n$. Then, \begin{itemize}
\item $\dep(w) = \ell_S(w)$ if and only if $w$ avoids the pattern $321$, and
\item $\ell_T(w) = \dep(w)$ if and only if $w$ avoids the patterns $321$ and $3412$.
\end{itemize}
In particular, there are Catalan-many permutations for which $\dep(w) = \ell_S(w)$, and odd-indexed-Fibonacci-many permutations for which $\ell_T(w) = \dep(w)$.
\end{thm}

The case $\ell_T(w) = \dep(w)$ also turns out to imply $\ell_T(w)=\ell_S(w)$, as discussed in Observation \ref{obs:2}. The permutations for which $\ell_T(w) = \ell_S(w)$ were first characterized by Edelman \cite[Theorem 3.1]{edelman}. Here we see Edelman's ``unimodal" permutations in a slightly different guise, following work of the second author \cite{tenner patt-bru, tenner rdpp}, which provides a dictionary for translating between one-line notation of a permutation and reduced decompositions. The permutations for which $\dep(w) = \ell_S(w)$ are characterized directly; we also have a bijection with Dyck paths to help make the characterization more intuitive.

The in-depth discussion of the case of permutations begins in Section \ref{sec:perm}. In particular, Theorem \ref{thm:char} is proved in Section \ref{sec:char} and Theorem \ref{thm:equidist} is proved in Section \ref{sec:equidist}. Theorem \ref{thm:extremes} is a summary of Theorems~\ref{thm:hequalsl} and~\ref{thm:equivalent lengths}, proved in Sections \ref{sec:hequalsl} and \ref{sec:lequalsl}, respectively.

As motivation for comparing the three statistics $\ell_T$, $\dep$, and $\ell_S$, we compare their values for elements of $\mf{S}_3$ and $\mf{S}_4$ in Table \ref{tab:llh}.

\begin{table}
$$
\begin{array}{c}
\begin{array}{c||c|c|c}
w \in \mf{S}_3 =A_2 & \ell_T(w) & \dep(w) & \ell_S(w)\\
\hline
123 & 0 & 0 & 0\\
213 & 1 & 1 & 1\\
132 & 1 & 1 & 1\\
312 & 2 & 2 & 2\\
231 & 2 & 2 & 2\\
321 & 1 & 2 & 3
\end{array} \\
\\
\begin{array}{c||c|c|c}
w \in G_2 & \ell_T(w) & \dep(w) & \ell_S(w)\\
\hline
e & 0 & 0 & 0\\
s_1 & 1 & 1 & 1\\
s_2 & 1 & 1 & 1\\
s_1s_2 & 2 & 2 & 2\\
s_2s_1 & 2 & 2 & 2\\
s_1s_2s_1 & 1 & 2 & 3\\
s_2s_1s_2 & 1 & 2 & 3\\
s_1s_2s_1s_2 & 2 & 3 & 4\\
s_2s_1s_2s_1 & 2 & 3 & 4\\
s_1s_2s_1s_2s_1 & 1 & 3 & 5\\
s_2s_1s_2s_1s_2 & 1 & 3 & 5\\
s_1s_2s_1s_2s_1s_2 & 2 & 4 & 6
\end{array}
\end{array}
\qquad
\begin{array}{c||c|c|c}
w \in \mf{S}_4 & \ell_T(w) & \dep(w) & \ell_S(w)\\
\hline
1234 & 0 & 0 & 0\\
2134 & 1 & 1 & 1\\
1324 & 1 & 1 & 1\\
1243 & 1 & 1 & 1\\
2314 & 2 & 2 & 2\\
2143 & 2 & 2 & 2\\
3124 & 2 & 2 & 2\\
1342 & 2 & 2 & 2\\
1423 & 2 & 2 & 2\\
3214 & 1 & 2 & 3\\
1432 & 1 & 2 & 3\\
2341 & 3 & 3 & 3\\
2413 & 3 & 3 & 3\\
3142 & 3 & 3 & 3\\
4123 & 3 & 3 & 3\\
3241 & 2 & 3 & 4\\
2431 & 2 & 3 & 4\\
4132 & 2 & 3 & 4\\
4213 & 2 & 3 & 4\\
3412 & 2 & 4 & 4\\
4231 & 1 & 3 & 5\\
4312 & 3 & 4 & 5\\
3421 & 3 & 4 & 5\\
4321 & 2 & 4 & 6
\end{array}
$$
\caption{Reflection length, depth, and length for the elements of the groups $\mf{S}_3=A_2$, $G_2$, and $\mf{S}_4=A_3$.  Note that each element $w$ satisfies the inequalities $(\ell_T(w)+\ell_S(w))/2 \le \dep(w) \le \ell_S(w)$.}\label{tab:llh}
\end{table}

\section{Depth in a Coxeter group}\label{sec:depth}

We assume some basic knowledge of Coxeter groups in this section. See Humphreys \cite{H} for facts about Coxeter groups that are not explained here. The reader interested only in permutation statistics can safely skip to Section \ref{sec:perm}.

\subsection{Depth in terms of words}

Let $W$ be a Coxeter group, with simple reflection set $S$. The set of all reflections is the set of conjugates of the simple reflections, denoted $T = \{ wsw^{-1}: w \in W, s \in S\}$. There are two common measures associated to any element $w \in W$. Its \emph{length}, $\ell_S(w)$, is the minimal number of simple reflections needed to express $w$, while its \emph{reflection length}, $\ell_T(w)$, is the minimal number of not-necessarily-simple reflections:
\begin{align*}
 \ell_S(w) &= \min\{k : w = s_{1}\cdots s_{i}, s_i \in S\}, \\
  \ell_T(w) &= \min\{ k : w = t_1 \cdots t_k, t_i \in T\}.
\end{align*}

Decompositions of the form $w = s_1\cdots s_{\ell_S(w)}$ with $s_i \in S$ or $w = t_1 \cdots t_{\ell_T(w)}$ with $t_i \in T$ are called \emph{reduced decompositions} (with respect to $S$ or with respect to $T$).

Now let $\Phi = \Pi \cup -\Pi$ denote the root system associated to $W$, and let $\Delta=\{ \alpha_1, \alpha_2, \ldots \}$ denote the \emph{simple roots}. Following \cite{BB} (see also \cite{stemQuo}), the \emph{depth} of a positive root $\beta \in \Pi$, denoted $\dep(\beta)$, is \[ \dep(\beta) = \min\{ k : s_1\cdots s_k(\beta) \in -\Pi,  s_i \in S\}.\] It is always the case that $\dep(\beta)\geq 1$, with $\dep(\beta) = 1$ if and only if $\beta \in \Delta$ is a simple root.

There is a one-to-one correspondence between reflections and positive roots: the reflection $t_{\beta}$ is the reflection across the hyperplane whose normal vector is $\beta$. In particular, $t_{\beta}(\beta) = -\beta$. Now, for any group element $w \in W$, we define its depth to be the minimum sum of depths in a factorization of $w$ into reflections. That is, \[ \dep(w) = \min \left\{ \sum_i \dep(\beta_i) : w = t_{\beta_1} \cdots t_{\beta_k}, t_{\beta_i} \in T\right\}.\]

We can see that, by definition, \[ 2\dep(\beta)=\ell_S(t_{\beta})+1.\] It then follows that \[ \dep(t_{\beta}) = \dep(\beta) = \frac{\ell_S(t_{\beta}) + 1}{2}, \] since if \[ t_{\beta} = t_{\gamma_1}\cdots t_{\gamma_k},\] is any factorization of $t_{\beta}$ as a product of reflections, then 
\begin{align*} 
2(\dep(\gamma_1) + \cdots +\dep(\gamma_k)) &= \ell_S(t_{\gamma_1}) + \cdots + \ell_S(t_{\gamma_k}) + k,\\
&\geq \ell_S(t_{\beta}) + k,\\
&\geq 2\dep(\beta) + k-1,\\
&\geq 2\dep(\beta).
\end{align*}
In fact, this shows that any expression for $t_{\beta}$ using more than one reflection forces a factorization of greater depth.

We can now describe depth solely in terms of the lengths of the reflections in a reflection factorization of $w$.

\begin{obs}\label{obs:dl}
The depth of $w$ is given by 
\begin{align*}
 \dep(w) &= \min\left\{ \sum_{i=1}^k \dep(t_i) : w = t_1\cdots t_k, \text{ where } t_i \in T \right\}, \\
 &= \frac{1}{2}\min\left\{ k + \sum_{i=1}^k \ell_S(t_i) : w = t_1\cdots t_k, \text{ where } t_i \in T\right\}.
 \end{align*}
\end{obs}

From any reflection factorization $w = t_1\cdots t_k$, we get $\ell_T(w) \leq k$ and $\ell_S(w) \leq \sum_{i=1}^k \ell_S(t_i)$.  Thus Observation~\ref{obs:dl} gives $\dep(w) \geq (\ell_T(w) + \ell_S(w))/2$.  Moreover, if $w= s_1\cdots s_{\ell_S(w)}$ is a reduced factorization in $S$, then $\dep(w) \leq (\ell_S(w) + \sum_{i=1}^{\ell_S(w)} \ell_S(s_i))/2 = \ell_S(w)$. In summary, we have the following.

\begin{obs}\label{obs:1}
The depth of an element $w$ satisfies \[ \frac{\ell_T(w)+\ell_S(w)}{2} \leq \dep(w) \leq \ell_S(w).\]
\end{obs}

Diaconis and Graham prove the equivalent result for total displacement of a permutation \cite[Theorem 2]{DG}. See also \cite[Problem 5.1.1.28]{K}.

Another easy observation follows from noticing that a factorization of $w$ can never have fewer than $\ell_T(w)$ terms, and the minimum depth of a term is $1$. Thus if depth and reflection length are equal for $w$, then the depth of every term is actually equal to $1$. In other words, the factorization consists entirely of simple reflections, and thus length equals reflection length. Together with Observation \ref{obs:1}, we have the following.

\begin{obs}\label{obs:2}
For $w \in W$, we have $\dep(w) = \ell_T(w)$ if and only if $\ell_S(w) = \ell_T(w).$
\end{obs}

\begin{remark}[Depth is additive]
If $W = W_1 \times W_2$ is reducible, then all the reflections in $W_1$ commute with the reflections in $W_2$. Thus if $w = uv$ with $u \in W_1$, $v \in W_2$, then $\dep(w) = \dep(u)+\dep(v)$. Hence, to characterize depth, it suffices to study irreducible Coxeter groups.
\end{remark}

\begin{remark}[Depth versus height]
Recall that every root is a linear combination of simple roots, $\sum_i c_i \alpha_i$, and the \emph{height} of a root is $\sum_i c_i$. Thus, one can then assign a minimum ``height" to any reflection factorization just as we have assigned a depth. 

Whenever $\Phi$ is a finite crystallographic root system, depth and height agree for short roots. If $\Phi$ is a root system of type $A_n$, then all the roots are the same length and so the depth of a root is the same as its height. Thus, we could just as easily have called the permutation statistic we study in Sections~\ref{sec:perm} and~\ref{sec:coincidence} the ``height" of a permutation.  

However, if $\Phi$ has both long and short roots, then depth and height need not agree. Indeed, we believe that depth is the better behaved of the two statistics, since it need not even be the case that the height of a reflection $t$ be the same as the height of its corresponding positive root. (The reflection $s_1s_2s_1$ in $G_2$ corresponds to a root of height 4, for example.) Moreover, depth makes sense for all root systems, not only crystallographic ones. Observation~\ref{obs:dl} shows that depth is independent of the choice of root system; that is, it depends only on the Coxeter system $(W,S)$.
\end{remark}

\subsection{Depth in terms of the Bruhat graph}

Length, reflection length, and depth, can all be understood in terms of the \emph{Bruhat graph} of a Coxeter group. The Bruhat graph of $W$ is the graph whose vertices are elements of $W$, with a directed edge $u \to v$ if and only if $v = ut$ for some $t \in T$ and $\ell_S(u) < \ell_S(v)$. We will augment this graph slightly by assigning the weight $\dep(\beta)$ to the edge $u \to ut_{\beta}$. The edge-weighted Bruhat graphs of $A_2 = \mf{S}_3$ and $G_2$ are presented in Figure~\ref{fig:Bruhatgraphs}.

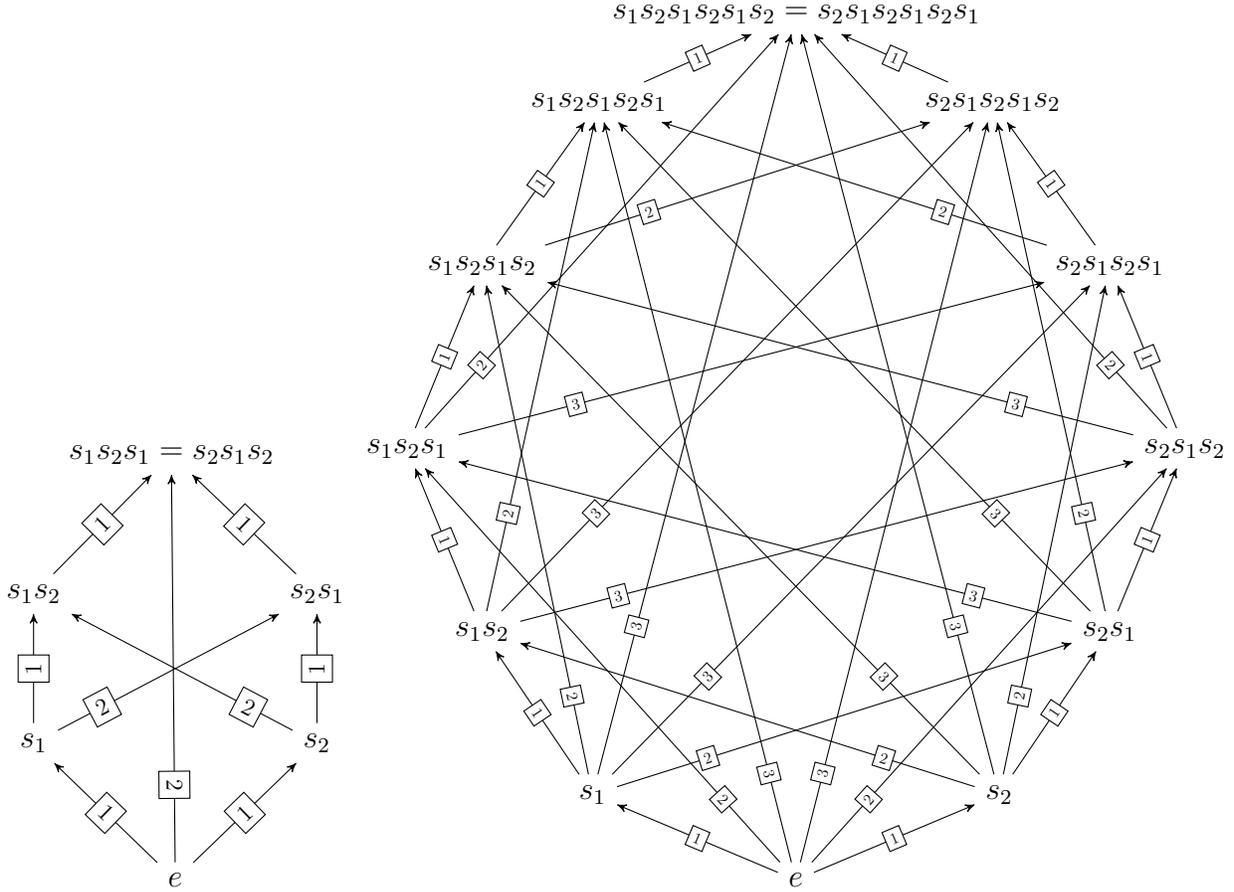
\begin{figure}
\[\begin{tikzpicture}[node distance=.45cm,>=stealth',bend angle=30]
\tikzstyle{state}=[draw=white,minimum size=4mm]
\tikzstyle{state2}=[draw,fill=white,sloped,scale=.75]
\node[state] (123) {$e$};
\node[state] (213) [above left=of 123,xshift=-10mm,yshift=10mm] {$s_1$};
\node[state] (132) [above right=of 123,xshift=10mm,yshift=10mm] {$s_2$};
\node[state] (312) [above=of 213,yshift=10mm] {$s_1s_2$};
\node[state] (231) [above=of 132,yshift=10mm] {$s_2s_1$};
\node[state] (321) [above right=of 312,xshift=-5mm,yshift=10mm] {$s_1s_2s_1=s_2s_1s_2$};
\path[->] (123) edge node[state2] {$1$} (213);
\path[->] (123) edge node[state2] {$1$} (132);
\path[->] (123) edge node[pos=.2,state2] {$2$} (321);
\path[->] (213) edge node[state2] {$1$} (312);
\path[->] (213) edge node[pos=.2,state2]  {$2$} (231);
\path[->] (132) edge node[pos=.2,state2] {$2$} (312);
\path[->] (132) edge node[state2] {$1$} (231);
\path[->] (312) edge node[state2] {$1$} (321);
\path[->] (231) edge node[state2] {$1$} (321);
\end{tikzpicture}
\begin{tikzpicture}[node distance=.9cm,>=stealth']  
\tikzstyle{state}=[draw=white,minimum size=4mm]
\tikzstyle{state2}=[draw,fill=white,sloped,scale=.5]
\node[state] (e) {$e$};
\node[state] (1) [above left=of e,xshift=-15mm] {$s_1$};
\node[state] (2) [above right=of e,xshift=15mm] {$s_2$};
\node[state] (12) [above left=of 1,yshift=10mm] {$s_1s_2$};
\node[state] (21) [above right=of 2,yshift=10mm] {$s_2s_1$};
\node[state] (121) [above=of 12,xshift=-10mm,yshift=10mm] {$s_1s_2s_1$};
\node[state] (212) [above=of 21,xshift=10mm,yshift=10mm] {$s_2s_1s_2$};
\node[state] (1212) [above=of 121,xshift=10mm,yshift=10mm]{$s_1s_2s_1s_2$};
\node[state] (2121) [above=of 212,xshift=-10mm,yshift=10mm] {$s_2s_1s_2s_1$};
\node[state] (12121) [above right=of 1212,xshift=-10mm, yshift=10mm] {$s_1s_2s_1s_2s_1$};
\node[state] (21212) [above left=of 2121,xshift=10mm, yshift=10mm] {$s_2s_1s_2s_1s_2$};
\node[state] (121212) [above right=of 12121,xshift=-16.5mm] {$s_1s_2s_1s_2s_1s_2 = s_2s_1s_2s_1s_2s_1$};
\path[->] (e) edge node[state2] {$1$} (1);
\path[->] (e) edge node[state2] {$1$} (2);
\path[->] (e) edge node[pos=.16,state2] {$2$} (121);
\path[->] (e) edge node[pos=.16,state2] {$2$} (212);
\path[->] (e) edge node[pos=.12,state2] {$3$} (12121);
\path[->] (e) edge node[pos=.12,state2] {$3$} (21212);
\path[->] (1) edge node[state2] {$1$} (12);
\path[->] (1) edge node[pos=.2,state2] {2} (21);
\path[->] (1) edge node[pos=.16,state2] {2} (1212);
\path[->] (1) edge node[pos=.2,state2] {3} (2121);
\path[->] (1) edge node[pos=.2,state2] {3} (121212);
\path[->] (2) edge node[state2] {$1$} (21);
\path[->] (2) edge node[pos=.16,state2] {2} (2121);
\path[->] (2) edge node[pos=.2,state2] {2} (12);
\path[->] (2) edge node[pos=.2,state2] {3} (121212);
\path[->] (2) edge node[pos=.2,state2] {3} (1212);
\path[->] (12) edge node[state2] {1} (121);
\path[->] (12) edge node[pos=.2,state2] {2} (12121);
\path[->] (12) edge node[pos=.2,state2] {3} (21212);
\path[->] (12) edge node[pos=.16,state2] {3} (212);
\path[->] (21) edge node[state2] {1} (212);
\path[->] (21) edge node[pos=.2,state2] {2} (21212);
\path[->] (21) edge node[pos=.16,state2] {3} (121);
\path[->] (21) edge node[pos=.2,state2] {3} (12121);
\path[->] (121) edge node[state2] {1} (1212);
\path[->] (121) edge node[pos=.16,state2] {2} (121212);
\path[->] (121) edge node[pos=.2,state2] {3} (2121);
\path[->] (212) edge node[state2] {1} (2121);
\path[->] (212) edge node[pos=.2,state2] {3} (1212);
\path[->] (212) edge node[pos=.16,state2] {2} (121212);
\path[->] (1212) edge node[state2] {1} (12121);
\path[->] (1212) edge node[pos=.27,state2] {2} (21212);
\path[->] (2121) edge node[state2] {1} (21212);
\path[->] (2121) edge node[pos=.27,state2] {2} (12121);
\path[->] (12121) edge node[state2] {1} (121212);
\path[->] (21212) edge node[state2] {1} (121212);
\end{tikzpicture}
\]
\caption{The edge-weighted Bruhat graphs of type $A_2$ and $G_2$.}\label{fig:Bruhatgraphs}
\end{figure}

By construction, any directed path from the identity $e$ to $w$ defines a reflection factorization of $w$. While it is not quite immediate from definitions, Dyer \cite{Dyer} shows that $\ell_T(w)$ is equal to the minimal number of edges in such a path. (What is clear is that $\ell_T(w)$ is the minimal number of edges in an \emph{undirected} path.) More easily shown is that the maximal length of a directed path from the identity to $w$ is $\ell_S(w)$.

In terms of the Bruhat graph, then, the depth of an element $w$ is the minimum weight of an undirected path from the identity $e$ to $w$ in the weighted Bruhat graph. Moreover, it appears that we can choose the path to be a directed path that has minimal length, $\ell_T(w)$.

\begin{ques}\label{q:path}
Is it true that \[\dep(w) = \min\left\{\sum_{i=1}^k d_i : \mbox{ there is a path } e \xrightarrow{d_1} \cdots \xrightarrow{d_k} w \mbox{ in the Bruhat graph}\right\}? \] And if so, can the path be chosen so that it has  $\ell_T(w)$ edges?
\end{ques}

The answer is yes for the symmetric group, as we will demonstrate with an algorithm in Section \ref{sec:char}, and it is easily verified in the dihedral case as we now discuss.

\subsection{Depth for dihedral groups}

For dihedral groups, depth is straightforward. Let $I_2(m)$ denote the dihedral group of order $2m$, for $m \leq \infty$.  Let $S = \{s_1,s_2\}$ denote the simple reflections.

\begin{prop}
For an element $w \in I_2(m)$, we have
$$ \dep(w) = \left\lceil \frac{\ell_S(w)+1}{2}\right\rceil.$$
Hence, \[ \sum_{w \in I_2(m)} q^{\ell_S(w)}t^{\dep(w)} =\begin{cases}
\displaystyle 1+2qt+q^mt^{\frac{m}{2} +1} + 2(1+q)t\sum_{i=1}^{\frac{m}{2}-1} q^{2i}t^i & \mbox{if $m$ is even,}\\
\displaystyle 1+2qt+q^{m-1}t^{\frac{m+1}{2}}(2+q) + 2(1+q)t\sum_{i=1}^{\frac{m-3}{2}} q^{2i}t^i & \mbox{if $m$ is odd, and}\\
\displaystyle 1+2qt\cdot\frac{1+qt}{1-q^2t} & \mbox{if $m=\infty$.}
\end{cases}
\]
\end{prop}

\begin{proof}
Indeed, suppose an element $w = s_1 s_2 s_1 \cdots$. Then if $\ell_S(w)$ is odd, then $w \in T$ and $\dep(w) = (\ell_S(w)+1)/2$.  Similarly, for even $\ell_S(w)$, we must have $\dep(w) =\ell_S(w)/2 + 1$: factoring into two reflections $(s_1)(s_2s_1\cdots s_2)$ gives this depth, and factoring into more than two reflections necessarily increases depth.

The generating function expressions follow, since there are exactly two elements of each length, apart from the identity (length zero) and the long element (length $m$).
\end{proof}

For example, the distributions of length and height for $I_2(3) = A_2$ and $I_2(6)=G_2$ are given in Table \ref{tab:llh}.

We now turn our focus to the case of the symmetric group $\mf{S}_n$; that is, the Coxeter group of type $A_{n-1}$.

\section{Depth in the symmetric group}\label{sec:perm}

For a positive integer $n$, let $[n] = \{1, 2, \ldots, n\}$, and let $\mf{S}_n$ be the symmetric group  on $[n]$; that is, the set $\{w: [n]\to [n]\}$ of all bijections from $[n]$ to itself. This group is a Coxeter group of type $A_{n-1}$.  The set of simple reflections in $\mf{S}_n$ is denoted by $S_n = \{s_1, \ldots, s_{n-1}\}$, where $s_i = (i\, i+1) \in \mf{S}_n$ is the adjacent transposition interchanging $i$ and $i+1$, and fixing all other elements.

The reflections in $\mf{S}_n$ are the transpositions. We denote the set of all transpositions by 
$$T_n = \{t_{ij} : 1 \le i < j \le n\},$$
where $t_{ij} = (i\, j) \in \mf{S}_n$ is the permutation interchanging $i$ and $j$, and fixing all other elements.  To avoid confusion with the subscripts, we will occasionally write $t_{i,j}$ for $t_{ij}$.  Notably, $s_i = t_{i,i+1}$. It is easily checked that the depth of a transposition is \[\dep(t_{ij}) = j-i.\]

We will freely write elements of $w \in \mf{S}_n$ as permutations in one-line notation, permutations in cycle notation, words on the set $S_n$, or words on the set $T_n$. For example,
$$2431756 = (124)(3)(576) = s_1s_2s_3s_2s_6s_5 = t_{12}t_{24}t_{67}t_{56}.$$

It is well known that length is equal to the number of \emph{inversions} of $w$. Let $\inv(w)$ denote the number of pairs $i < j$ such that $w(i) > w(j)$. For example, we see that $\inv(2431756) = 6$. Also well known is that reflection length is equal to $n-c(w)$, where $c(w)$ denotes the number of cycles of $w$. With $w= (124)(3)(576)$ as above, we see that $7-c(w) = 4$. It is interesting that length is easily computed from the one-line notation of a permutation, whereas reflection length is easily computed from cycle notation. It would be very interesting if there was a notation for permutations in which both length and reflection length are easily seen.

\subsection{An algorithm and a characterization of depth}\label{sec:char}

We begin our discussion of depth of a permutation by reconsidering straight selection sort applied to $w=3715246$. We have:
\[\begin{array}{c c c} &3\bm{7}15246 & \\
 (67)\cdot & \downarrow  & \cdot(27) \\
& 3\bm{6}15247 & \\
(46)\cdot &  \downarrow & \cdot(26) \\
 & 341\bm{5}267 & \\
(25)\cdot & \downarrow & \cdot(45) \\
 & 3\bm{4}12567 & \\
(24)\cdot & \downarrow & \cdot(24) \\
 & \bm{3}214567 & \\
(13)\cdot & \downarrow & \cdot(13) \\
 & 1234567 &
\end{array}
 \]
In each line, we have written on the left the transposition $( w(i)\, i)$ for the highlighted letter $i$ (the largest number not yet in its proper place), while on the right we have $( w^{-1}(i)\, i)$. The difference is only whether we consider the action of left multiplication, which swaps the specified digits, or right multiplication, which swaps the digits in the specified positions.
 
We see that we could associate many different factorizations of $w$ to the straight selection sort algorithm. For the sorting index of \cite{petersen 1}, we use right multiplication at each step to obtain a reduced reflection decomposition for $w$. This gives $w = t_{13}t_{24}t_{45}t_{26}t_{27}$, and \[\dep(w) \leq (3-1) + (4-2) + (5-4) + (6-2) + (7-2) = 14.\] On the other hand, we can see equally well that $t_{13}t_{24}t_{46}t_{67}\cdot w \cdot t_{45}=e$, and so $w = t_{67}t_{46}t_{24}t_{13}t_{45}$. This factorization shows \[ \dep(w) \leq (7-6) + (6-4) + (4-2) + (3-1) + (5-4) = 8.\]

In fact, by choosing the transposition of least depth in each step of the straight selection sort, we achieve a minimal depth factorization. We make the procedure precise with the algorithm \textsf{SHALLOW-DECOMP} below.

\begin{quote}
\noindent \textbf{Algorithm} \textsf{SHALLOW-DECOMP.}\\
The input of the algorithm is a permutation $w \in \mf{S}_n$. The output is a pair of reduced decompositions $u = u(w)$, and $v = v(w)$, such that $w = uv$.
\begin{enumerate}
\item If $n=1$, then $u:=e$, $v:=e$.
\item If $w(n) = n$, then let $w':=w(1)\cdots w(n-1)$ and $u:=u(w')$, $v:=v(w')$.
\item If $w(n) < n$ then $w':=t_{w(n), n}\cdot w$ and:
\begin{enumerate}
\item if $w(n) \leq w^{-1}(n)$ then $u:=u(w')$, $v:= v(w')\cdot t_{w^{-1}(n), n}$,
\item if $w(n) > w^{-1}(n)$ then $u:=t_{w(n),n}\cdot u(w')$, $v:=v(w')$.
\end{enumerate} 
\end{enumerate}
\end{quote}

With the example of $w = 3715246$ as above, we can trace through the algorithm to find \[ u = t_{67}t_{46}t_{24}t_{13} \quad \mbox{ and } \quad v = t_{45}.\] We now prove the following, which establishes Theorem \ref{thm:char}.

\begin{thm}\label{thm:alg}
For $w \in \mf{S}_n$, the output of \textsf{SHALLOW-DECOMP} gives a reduced expression $w=uv$ of minimal reflection length and minimal depth. Moreover, this depth is equal to
$$\dep(w) = \sum_{w(i) > i} (w(i) - i),$$ and the reduced expression corresponds to a directed path in the Bruhat graph.
\end{thm}

\begin{proof}
The fact that $uv=w$ is obvious by construction. That the factorization has $\ell_T(w)$ terms follows from the fact that straight selection sort achieves minimal reflection length: a transposition can only change the number of cycles by $\pm 1$, and the algorithm uses steps that each increase the number of fixed points, and hence, cycles. 

To build a path in the Bruhat graph from the factorization into reflections, say $u=u_1u_2\cdots u_k$ and $v=v_l\cdots v_2v_1$, where the $u_i$ and the $v_j$ are transpositions. The indexing indicates the fact that the word $u$ is built from left to right in part 3(b) of the sorting, while $v$ is built from right to left in part 3(a). From steps 3(a), we find a path in the Bruhat graph $$wv_1\cdots v_l \to \cdots \to wv_1v_2 \to wv_1 \to w.$$ The fact that length increases with each edge is obvious from construction. (Remember, the sorting algorithm is putting the largest misplaced letter in its proper place.) Similarly, from steps 3(b), we find a path $$e \to u_1 \to u_1u_2 \to \cdots \to u_1\cdots u_k = wv_1\cdots v_l.$$ Thus the word $u$ builds the path from the bottom up and $v$ builds the path from the top down. They meet in the middle to form a path from $e$ to $w$. 

To see that the algorithm achieves the minimal depth, we proceed by induction on $n$. When $n=1$, this is obviously true, and the given formula for depth works.

Now suppose $n>1$ and let
$$d(w) = \sum_{w(i) > i} (w(i) -i).$$
Suppose $\dep(w') = d(w')$ for all $w' \in \mf{S}_{n-1}$. Let $w \in \mf{S}_n$. If $w(n) = n$, then we can view $w$ as lying in $\mf{S}_{n-1}$, and so $\dep(w)=d(w)$.

Now let $j = w^{-1}(n) < n$. We will first show $\dep(w)\leq d(w)$. As in step (2) of \textsf{SHALLOW-DECOMP}, we let $w' = t_{w(n),n}\cdot w$. There are two cases to consider, given in steps (3a) and (3b). In the first case, from (3a), we suppose that $w(n) \leq j$, in which case $w'(j) = w(n)$ is not an excedance of $w'$, so
\begin{align*}
 d(w) &= \sum_{w(i) > i} (w(i)-i), \\
      &= (n-j) + \sum_{i\neq j : w(i) > i} (w(i)-i), \\
      &= (n-j) + \sum_{w'(i) > i} (w'(i)-i),\\
      &= (n-j) + d(w'),\\
      &= (n-j) + \dep(w').
\end{align*}
We get that $d(w)$ is the sum of the depths in the factorization 
\[w=u(w')\cdot v(w')\cdot t_{j n},\] and so $\dep(w) \leq d(w)$.

For the next case, from (3b) of the algorithm, we suppose that $w(n) = w'(j) > j$. This time,  
\begin{align*}
 d(w) &= \sum_{w(i) > i} (w(i)-i), \\
      &= (n-j) + \sum_{i\neq j : w(i) > i} (w(i)-i), \\
      &= (n-j) + \sum_{i\neq j : w'(i) > i} (w'(i)-i),\\
      &= (n-j) + d(w') - (w(n) -j),\\
      &= (n-w(n)) + d(w'),\\
      &= (n-w(n)) + \dep(w').
\end{align*}
Again we conclude that $d(w)$ is the sum of the depths in the factorization \[ w = t_{w(n),n}\cdot u(w')\cdot v(w'),\] yielding $\dep(w) \leq d(w)$.

Thus, we see that $d(w)$ is an upper bound for depth in all cases. On the other hand, we know $d(w) \leq \dep(w)$, because, at the very least, all of the excedances $w(i)$ in a permutation are $w(i) -i$ places away from their initial positions. Combining the inequalities we have $\dep(w) = d(w)$, as desired.
\end{proof}

We include here for reference Table \ref{tab:hnk}, which contains the distribution of depths for $n\leq 8$. This array can be found as entry A062869 of \cite{oeis}.

\begin{table}[h]
\begin{center}
{\tiny \begin{tabular}{ c |ccccccc c c c c c cc cc cccc}
& $k=0$ & 1 & 2 & 3 &  4 & 5 & 6 & 7 & 8 & 9 & 10 & 11 & 12 & 13 & 14 & 15 & 16\\
\hline
$n=1$ &1  \\

2 &1 & 1 \\

3 &1 & 2 & 3 && && &&  &&\\

4 &1 & 3 & 7 & 9 & 4 &&  &&\\

5 &1 & 4 & 12 & 24 & 35 & 24 & 20 \\

6 &1 & 5 & 18 & 46 & 93 & 137 & 148 & 136 & 100 & 36 \\

7 &1 & 6 & 25 & 76 & 187 & 366 & 591 & 744 & 884 & 832 & 716 & 360 & 252\\

8 & 1 & 7 & 33 & 115 & 327 & 765 & 1523 & 2553 & 3696 & 4852 & 5708 & 5892 & 5452 & 4212 & 2844 & 1764 & 576\\
\end{tabular}}
\end{center}
\caption{The number of permutations $w \in \mf{S}_n$ with depth $\dep(w) =k$.}\label{tab:hnk}
\end{table}

Given the formula for depth in Theorem \ref{thm:alg}, we can prove the following.

\begin{prop}\label{prp:upper}
For all $w \in \mf{S}_n$, we have $\dep(w) \leq \lfloor n^2/4\rfloor$, and this bound is sharp.
\end{prop}

Further, we can say precisely how many permutations achieve this upper bound.

\begin{prop}\label{prp:upper2}
The number of permutations in $\mf{S}_n$ achieving maximal depth is \[|\{ w \in \mf{S}_n : \dep(w) = \lfloor n^2/4\rfloor\}| = \begin{cases}(k!)^2 & \mbox{if $n=2k$,}\\
n(k!)^2 & \mbox{if $n=2k+1$}.
\end{cases}\]
\end{prop}

Statements equivalent to both these propositions can be found in the paper of Diaconis and Graham, although without proof (see Table 1 and Remark 2 of \cite{DG}).  They are also mentioned in the remarks (and links therein) for entry A062870 of \cite{oeis}. In particular, Alekseyev has a webpage \cite{alekseyev} with a proof of Proposition \ref{prp:upper2}.

To prove the propositions, it will be helpful to have the following lemma.

\begin{lem}\label{lem:exc}
Let $w \in \mf{S}_n$, and suppose $i< j$ and $w(i)<w(j)$. Then,
$$\dep(w\cdot t_{ij}) = \begin{cases} \dep(w) & \mbox{if $j < w(i)$ or $w(j) \leq i$,} \\
\dep(w)+\min\{w(j),j\}-\max\{w(i),i\} & \mbox{otherwise.}
\end{cases}$$
In particular, $\dep(w\cdot t_{ij}) \geq \dep(w)$ for all such $w$, $i$, and $j$.
\end{lem}

\begin{proof}
Since $\dep(w) = \frac{1}{2}\sum_{i=1}^n |w(i)-i|$, half of total displacement (see Remark \ref{rem:TD}), it will suffice to study the following quantity:
\begin{align*}
 \kappa &=\dep(w\cdot t_{ij})-\dep(w) \\
  &= \frac{1}{2}\left(|w(j)-i|+|w(i)-j|-|w(i)-i|-|w(j)-j|\right).
\end{align*}

There are six cases, and for each of these $\kappa$ is easily computed:
\[
\begin{array}{| c |c |}
\hline
\mbox{Case} & \mbox{Value of $\kappa$} \\
\hline \hline
w(i) < w(j) \leq i < j & 0 \\
w(i)\leq i < w(j) \leq j & w(j)-i \\
w(i) \leq i < j < w(j) & j-i \\
i < w(i) < w(j) \leq j & w(j)-w(i) \\
i < w(i) \leq j < w(j) & j-w(i) \\
i < j < w(i) < w(j) & 0 \\
\hline
\end{array}
\] 

For example, if $w(i)<w(j)\leq i<j$, then, 
\[
 2\kappa = (i-w(j)) + (j-w(i)) - (i-w(i)) -(j -w(j)) = 0,
\] 
while if $w(i)\leq i < w(j) \leq j$, then, 
\[
 2\kappa = (w(j)-i) + (j-w(i)) - (i-w(i)) -(j -w(j)) = 2(w(j)-i).
\] 

The values of $\kappa$ computed here prove the lemma. 
\end{proof}

With the lemma in hand, we now prove Proposition \ref{prp:upper}.

\begin{proof}[Proof of Proposition \ref{prp:upper}]
Observe that the longest permutation, $w_0 = n (n-1) \cdots 21$, achieves the claimed upper bound of $\lfloor n^2/4 \rfloor$. Indeed, if $n$ is even, then \begin{align*}
 \dep(w_0) = (n-1) + ((n-1)-2) + \cdots &= (n-1) + (n-3) +\cdots + 5+ 3+1 \\
 &= \left(\frac{n}{2}\right)^2.
 \end{align*}
If $n$ is odd, then
\begin{align*}
 \dep(w_0) = (n-1) + ((n-1)-2) + \cdots &= (n-1) + (n-3) +\cdots + 4+ 2 \\
 &= \left(\frac{n+1}{2}\right)\left(\frac{n+1}{2}-1\right) = \frac{n^2-1}{4}.
\end{align*}

Now, we claim that $\dep(w_0) \geq \dep(w)$ for any $w \in \mf{S}_n$.  

First suppose $w_0 = w t_{ij}$. Then since $w_0$ is totally decreasing, it must be that $i < j$ and $w(i) < w(j)$. Thus $\dep(w_0) \geq \dep(w)$ by Lemma \ref{lem:exc}.

In general, we can proceed from any permutation $w$ to $w_0$ by transpositions of this type. To see this, define  $A(w) = \{ t_{ij} : i < j, w(i) < w(j) \}$. Then $|A(w_0)| = 0$, and it is the only permutation with this property. Generally, if $t \in A(w)$, then
$$A(w t) \subseteq (A(w)\setminus\{t\}).$$
In particular, $|A(w t)| < |A(w)|$. So by some choice of transpositions $t_{ij}$ with $i< j$ and $w(i)<w(j)$, we can move from any $w$ to $w_0$. Lemma \ref{lem:exc} shows that this cannot decrease depth. So there exists a sequence of transpositions $t_1, \ldots, t_r$ such that
$$|A(w)|>|A(w t_1)|>\cdots > |A(w t_1\cdots t_r)| = |A(w_0)| = 0,$$
and
$$\dep(w_0) = \dep(w t_1\cdots t_r) \geq \cdots \geq \dep(wt_1) \geq \dep(w).$$

This completes the proof.
\end{proof}
 
Now we will prove Proposition \ref{prp:upper2}, characterizing the permutations with maximal depth.

\begin{proof}[Proof of Proposition \ref{prp:upper2}]
To begin, suppose $n=2k$ is even. We will show that the permutations for which $\dep(w) = k^2$ are precisely those permutations for which 
$$\{ w(1), w(2),\ldots,w(k)\} = \{k+1,k+2,\ldots,n\},$$
or, equivalently,
$$\{w(k+1),w(k+2),\ldots, w(n)\} = \{1,2,\ldots,k\}.$$
Denote this set by $\mf{S}_{k,k}$. Obviously there are $(k!)^2$ permutations in this set, since the first $k$ elements can be permuted independently of the final $k$ elements. 

Suppose $w \in \mf{S}_{k,k}$. If $i, j \in \{1,2,\ldots,k\}$ or if $i,j \in \{k+1,k+2,\ldots,n\}$, then $wt_{ij} \in \mf{S}_{k,k}$ as well. Moreover, from Lemma \ref{lem:exc}, we see that \[ \dep(wt_{ij}) = \dep(w).\]
To see that these depths are equal to $k^2 =\lfloor n^2/4\rfloor$, we note that, just as in the proof of Proposition \ref{prp:upper}, a succession of such transpositions can be used to obtain the decreasing permutation $w_0 \in \mf{S}_{k,k}$, whose depth was already established as $\dep(w_0) = n^2/4$.

Finally, we observe that if $w \notin \mf{S}_{k,k}$, then there are indices $i$ and $j$, with $i \leq k < j$, such that $w(i) \leq k < w(j)$. Using Lemma \ref{lem:exc} and Proposition \ref{prp:upper} we have $\dep(w)< \dep(wt_{ij}) \leq \dep(w_0)$.

The case when $n$ is odd proceeds analogously, except that we will define $n = 2k+1$ different sets $\mf{S}_{i;k,k}$, one for each choice $i \in [n]$ of the middle value, $w(k+1)$, of the permutation. 

If $i \leq k+1$, then $\mf{S}_{i;k,k}$ is the set of permutations for which
$$\{w(1),\ldots, w(k)\} = \{k+2,\ldots,n\},$$
or, equivalently,
$$\{w(k+2),\ldots,w(n)\} = \{1,\ldots,k+1\}-\{i\}.$$
If $i > k+1$, then $\mf{S}_{i;k,k}$ is the set of permutations for which
$$\{w(1),\ldots, w(k)\} = \{k+1,\ldots,n\}-\{i\},$$
or, equivalently,
$$\{w(k+2),\ldots,w(n)\} = \{1,\ldots,k\}.$$
Although $w_0 \notin \mf{S}_{i;k,k}$ if $i\neq k+1$, it is not hard to see that each $\mf{S}_{i;k,k}$ is closed under depth-preserving transpositions that do not involve $i$, and each $\mf{S}_{i;k,k}$ contains a ``nearly decreasing" element $w$ whose non-middle entries are mapped as
$$w:(1,2,\ldots,\widehat{k+1},\ldots,n) \mapsto (n,n-1,\ldots,\widehat{i},\ldots,1).$$
Moreover, this $w$ has depth 
$$\frac{n^2-1}{4}$$
if $i\le k+1$ and depth
$$\frac{n^2-1}{4} - 1(i - (k+1)) + (i - (k+1)) = \frac{n^2-1}{4}$$
if $i>k+1$.  There are $(k!)^2 $ permutations in each $\mf{S}_{i;k,k}$, and so a total of $n(k!)^2$ permutations $w \in \mf{S}_{2k+1}$ with $\dep(w) = \lfloor n^2/4 \rfloor$.

If $w \notin \mf{S}_{i;k,k}$ for any $i$, it follows that there are $i < k+1 < j$ such that $w(i) < k+1 < w(j)$, and so again using Lemma \ref{lem:exc} and Proposition \ref{prp:upper} we have $\dep(w)< \dep(wt_{ij}) \leq \dep(w_0)$.

This completes the proof.
\end{proof}

\subsection{Depth is equidistributed with descent drop}\label{sec:equidist}

We will now prove Theorem \ref{thm:equidist}. Recall that $\des(w) = |\{ i : w(i) > w(i+1)\}|$ and $\exc(w) = |\{ i : w(i) > i\}|$. In the introduction we defined the \emph{descent drop} of $w$ to be \[ \drop(w) = \sum_{w(i) > w(i+1)} (w(i) - w(i+1)),\] and Theorem \ref{thm:char} shows that \[ \dep(w) = \sum_{w(i) > i} (w(i) - i).\] Theorem \ref{thm:equidist} claims the following:
\begin{equation}\label{eq:equi2}
\sum_{w \in \mathfrak{S}_n} q^{\drop(w)} t^{\des(w)} = \sum_{w \in \mathfrak{S}_n} q^{\tiny \dep(w)} t^{\exc(w)}.
\end{equation}

The proof of equation~\eqref{eq:equi2} follows from a bijection of Steingr\'imsson that carries descents to excendances \cite[Appendix]{stein}, related to the ``transformation fondamentale'' of Foata and Sch\"utzenberger \cite{foata-schutzenberger}.  We follow the description of the bijection given in Theorem 51 of \cite{stein}.

Given $w \in \mf{S}_n$, we form $v\in \mf{S}_n$ with the following process. For each $j \in [n]$:
\begin{itemize}
\item if $w(j) > w(k)$ for some $k>j$, then \[ v(w(j+1)) = w(j),\]

\item otherwise, let $i<j$ be maximal such that $w(i) < w(j)$ (with $w(0)=0$), and declare \[ v(w(i+1)) = w(j). \]
\end{itemize}

We now let $\phi: \mf{S}_n \to \mf{S}_n$ be given by $\phi(w) = v$.

For example, let $w = 7213645$. In working from left to right in $w$:
 we put $w(1) = 7$ in position $w(2) =2$ of $v$ because it $7$ greater than something to its right in $w$; 
we put $w(2)=2$ in position $w(3)=1$ of $v$ because $2$ is greater than something to its right in $w$; 
 we put $w(3) = 1$ in position $w(1) =7$ of $v$ since it is \emph{not} greater than something to its right in $w$ and $w(0)=0$ is the rightmost number less than $w(3)$, and so on.
 Continuing, we ultimately achieve $v=2736541$. \[ \begin{array}{r| c c c c c c c} w & 7 & 2 & 1 & 3 & 6 & 4 & 5 \\
\hline v & v(1) & v(2) & v(3) & v(4) & v(5) & v(6) & v(7) \\ \hline
& & 7 & \\
 &2& 7 &\\
 &2& 7 & & & & & 1\\
& 2& 7 & 3& & & & 1 \\
 &2& 7 & 3& 6 & & & 1 \\
& 2& 7 & 3& 6 & & 4& 1 \\
& 2& 7 & 3& 6 & 5 & 4& 1 \end{array}
\]

That the map $\phi$ is well-defined and injective (hence bijective) follows from close inspection of its definition, and we will omit proof of this fact. See \cite[Section 4.1]{stein thesis}.

More important from our perspective is the following lemma that follows by construction of $\phi$.

\begin{lem}{\cite[Section 4.1]{stein}}\label{lem:dex}
Let $w \in \mf{S}_n$ and $v =\phi(w)$. Then $w(i) > w(i+1)$ is a descent of $w$ if and only if $w(i) = v(w(i+1))$ is an excedance of $v$. In particular, $\des(w) = \exc(v)$.
\end{lem}

In the example above, $w = \bm{7}\bm{2}13\bm{6}45$ has descent pairs $7>2$, $2>1$, and $6>4$, so that $\drop(w) = 5 + 1 + 2 =8$. On the other hand, $v = \bm{27}3\bm{6}541$ has excedances $2$ (in position 1), $7$ (in position 2), and $6$ (in position 4), so $\dep(v) = 1 + 5 + 2 = 8$.

From Lemma \ref{lem:dex} we can now prove Theorem \ref{thm:equidist}.

\begin{proof}[Proof of Theorem \ref{thm:equidist}]
It suffices to show that $\phi$ carries descents to excedances, and descent drops to depth. Lemma \ref{lem:dex} states that $\des(w) = \exc(\phi(w))$, and moreover, with $v = \phi(w)$,
\begin{align*}
\drop(w) &= \sum_{w(i)>w(i+1)} (w(i) - w(i+1))\\
 &= \sum_{w(i) > w(i+1)} (v(w(i+1)) - w(i+1)) \\
 &= \sum_{v(k) > k} (v(k) - k) \\
 &= \dep(v) = \dep(\phi(w)),
\end{align*}
as desired.
\end{proof}

\subsection{The distribution of depth}\label{sec:depthk}

There are obviously $n-1$ elements of depth 1 (the simple reflections) and it is also not difficult to see that for $n\geq 3$: \[|\{ w \in \mf{S}_n : \dep(w) = 2 \}| = \binom{n-1}{2} + 2(n-2) = \frac{(n+3)(n-2)}{2}. \] In terms of reduced decompositions, these are all elements of the form:
\begin{itemize}
\item $s_is_j$, with $1\leq i < j\leq n-1$, 
\item $s_{i+1}s_i$, with $1\leq i\leq n-2$,
\item $s_i s_{i+1} s_i$, with $1 \leq i \leq n-2$.
\end{itemize}

In general, however, we have no simple formula for the number of permutations in $\mf{S}_n$ that have depth (or descent drop) $k$. 

Recently, Guay-Paquet and the first author \cite{GP} proved the following continued fraction formula for the ordinary generating function for depth:
\begin{align}
 F(t,z) &= \sum_{n\geq 0}\sum_{w \in \mf{S}_n} t^{\dep(w)} z^n \nonumber \\
  &= \frac{1}{1 - \displaystyle\frac{z}
                    {1 - \displaystyle\frac{t z}
                    {1 - \displaystyle\frac{2t z}
                    {1 - \displaystyle\frac{2t^2 z}
                    {1 - \displaystyle\frac{3t^2 z}
                    {1 - \displaystyle\frac{3t^3 z}
                    {1 - \displaystyle\frac{4t^3 z}
                    {1 - \displaystyle\frac{4t^4 z}
                    {1 - \cdots}}}}}}}}} \label{eq:contfrac}
\end{align}
This expression allows us, for example, to verify the formulas in Table \ref{tab:depk}, which were earlier computed by Timothy Walsh through extensive computation \cite{walsh}. The formulas are polynomials of degree $k$ that hold for $n\geq k$.

\begin{table}
\[
\begin{array}{c | c}
k & |\{ w \in \mf{S}_n : \dep(w) =k \}| \\
\hline
0 & 1 \\
1 & n-1 \\
2 & \binom{n-2}{2} + 3(n-2) \\
3 & \binom{n-3}{3} + 6\binom{n-3}{2} + 9(n-3) \\
4 & \binom{n-4}{4} + 9\binom{n-4}{3} + 27\binom{n-4}{2} + 31(n-4) + 4 \\
5 & \binom{n-5}{5} + 12\binom{n-5}{4} + 54\binom{n-5}{3} + 116\binom{n-5}{2} + 113(n-5) + 24\\
6 & \binom{n-6}{6} + 15\binom{n-6}{5} + 90\binom{n-6}{4} + 282\binom{n-6}{3} + 489\binom{n-6}{2} + 443(n-6) + 148 \\
7 & \binom{n-7}{7} + 18\binom{n-7}{6} + 135\binom{n-7}{5} + 556\binom{n-7}{4} + 1375\binom{n-7}{3} + 2074\binom{n-7}{2} + 1809(n-7) + 744
\end{array}
\]
\caption{The number of permutations with depth $k$.}\label{tab:depk}
\end{table}

The key ingredient in the proof of Equation \ref{eq:contfrac} involves a map from permutations to Motzkin paths that takes depth to the area under the path. This map makes Propositions \ref{prp:upper} and \ref{prp:upper2} easy corollaries, since it is easy to characterize which permutations map to the unique Motzkin path of maximal area. See \cite[Remarks 3.3, 4.3]{GP}.

\section{Coincidences of depth, length, and reflection length}\label{sec:coincidence}

In this section and the next, we will establish Theorem \ref{thm:extremes}, in which we characterize, in terms of pattern avoidance, those permutations for which $\dep(w) = \ell_S(w)$ and those for which $\dep(w)=\ell_T(w)$. The results we prove here were stated by Diaconis and Graham \cite[Remarks 3 and 4]{DG}, though they omitted proofs. Further, while they enumerated the $\dep(w) = \ell_T(w)$ case, they did not provide a characterization.

First, recall the notion of pattern avoidance. Let $w \in \mf{S}_n$ and $p \in \mf{S}_k$, where $n \ge k$. We say that $w$ \emph{contains} a $p$-pattern if there exist indices $\{i_1 < \cdots < i_k\}$ such that $w(i_1)\cdots w(i_k)$ is in the same relative order as $p(1)\cdots p(k)$.  Otherwise, we say that $w$ \emph{avoids} $p$, or is \emph{$p$-avoiding}. For example, the permutation $3241576$ contains the pattern $1234$ (in positions $\{1,3,5,6\}$, for example), and avoids the pattern $4321$.

\subsection{When depth equals length}\label{sec:hequalsl}

We now present a characterization of those permutations for which $\dep(w) = \ell_S(w)$.

\begin{thm}\label{thm:hequalsl}
For any $w$, $\dep(w) = \ell_S(w)$ if and only if $w$ avoids $321$.
\end{thm}

\begin{cor}
The number of $w \in \mf{S}_n$ for which $\dep(w) = \ell_S(w)$ is given by the Catalan number $\Cat_n = \frac{1}{n+1}\binom{2n}{n}$.
\end{cor}

Recall that a permutation $w$ is called \emph{fully commutative} if every reduced expression $w=s_1\cdots s_{\ell_S(w)}$ for $w$ (with simple reflections) can be obtained from any other by a sequence of swaps of adjacent letters: $s_i s_j = s_j s_i$ with $|i-j| > 1$. In \cite[Theorem 2.1]{BJS} it was shown that $w$ is fully commutative if and only if $w$ avoids 321. (See \cite{stem} for generalizations.) Thus, we obtain another characterization of when length equals depth.

\begin{cor}\label{cor:hlfc}
A permutation $w \in \mf{S}_n$ has $\dep(w) = \ell_S(w)$ if and only if $w$ is fully commutative.
\end{cor}

We will prove Theorem \ref{thm:hequalsl} by exhibiting a bijection between $\{w : \dep(w) = \ell_S(w)\}$ and Dyck paths. More specifically, we will demonstrate that the fibers of a certain map from $\mf{S}_n$ to Dyck paths of length $2n$ have unique minimal length representatives, and that these representatives are the 321-avoiding permutations.

Let $\Dy_n$ denote the set of Dyck paths of length $2n$; that is, those lattice paths from $(0,0)$ to $(n,n)$ that take steps North and East, and never pass below the line $y=x$.

Let $\lr(w)$ denote the set of \emph{left-right maxima} of $w$, written as pairs $(i,w(i))$; that is,
$$\lr(w) = \{ (i,w(i)) : w(j)<w(i) \mbox{ for all } j <i \}.$$
From $w$ we can then form a Dyck path by putting the outer corners of the path at coordinates $(i-1,w(i))$ for each left-right maximum. This is easiest to understand with an example. Let $w = 23176845$. We will draw $w$ as a collection of non-attacking rooks on a chessboard, where the value $w(i)$ is placed in the box appearing in the $i$th column from the left and the $w(i)$th row from the bottom. We circle the left-right maxima, then draw the Dyck path with these positions as outer corners. Let $P: \mf{S}_n \to \Dy_n$ denote this function. We see $P(23176845)$ in Figure \ref{fig:Dyck}.

\begin{figure}
\begin{tikzpicture}[scale=.75,node distance=.5cm,>=stealth',bend angle=30,auto]
\draw[style=help lines,step=1cm] (0,0) grid (8,8);
\draw (.5,1.5) node[circle,draw] {2};
\draw (.5,.5) node {$\bullet$};
\draw (1.5,2.5) node[circle,draw] {3};
\draw (1.5,1.5) node {$\bullet$};
\draw (2.5,.5) node {1};
\draw (3.5,6.5) node[circle,draw] {7};
\draw (3.5,3.5) node {$\bullet$};
\draw (3.5,4.5) node {$\bullet$};
\draw (3.5,5.5) node {$\bullet$};
\draw (4.5,5.5) node {6};
\draw (5.5,7.5) node[circle,draw] {8};
\draw (5.5,5.5) node {$\bullet$};
\draw (5.5,6.5) node {$\bullet$};
\draw (6.5,3.5) node {4};
\draw (7.5,4.5) node {5};
\draw[very thick] (0,0) -- (0,2) -- (1,2) -- (1,3) -- (3,3) -- (3,7) -- (5,7) -- (5,8) -- (8,8);
\end{tikzpicture}
\caption{An example of the map from $\mf{S}_n$ to $\Dy_n$.}\label{fig:Dyck}
\end{figure}

It will be useful to have the following lemma.

\begin{lem}\label{obs:fiber}
Fix $w \in \mf{S}_n$. If $w(i) - i > 0$, then \[|\{ j : i < j \mbox{ and } w(i) > w(j) \}| \geq w(i) - i,\] with equality if and only if $(i,w(i))$ is a left-right maximum. In particular,
\begin{eqnarray*} 
|\{ (i,j) : i < j, w(i)>w(j),  \mbox{ and } (i,w(i)) \in \lr(w)\}| &=& \sum_{(i,w(i)) \in \lr(w)} (w(i)-i) \\
&\leq& \dep(w)\\
&\leq& \ell_S(w).
\end{eqnarray*} 
\end{lem}

\begin{proof}
Every left-right maximum $(i,w(i))$ is the greater value in $w(i)-i$ inversions, since there are only $n-w(i)$ numbers greater than $w(i)$ and each of these must be placed among the $n-i$ positions $\{ i+1, i+2,\ldots,n\}$. So there are $n-i-(n-w(i)) = w(i)-i$ numbers $j$ such that $i < j$ and $w(i)>w(j)$.

If $(i,w(i))$ is not a left-right maximum, then there are strictly fewer than $n-w(i)$ numbers greater than $w(i)$, to be placed among the $n-i$ positions to its right. Thus, the number of positions $j$ such that $i < j$ and $w(i) > w(j)$ is:
\[ n-i - |\{\mbox{numbers larger than $w(i)$ to the right of $w(i)$}\}| > n-i-(n-w(i)) = w(i) -i.\]
\end{proof}

We have highlighted with bullet points the spaces below left-right maxima and on or above the main diagonal. Each bullet point can be identified with an inversion pair in which the left-right maximum above the bullet point is the greater value in the pair, and the value connected to the bullet point by a dashed line is the smaller value in the pair.

Let $p \in \Dy_n$ be a Dyck path, and let $P^{-1}(p)$ denote its preimage under $P$, that is, \[ P^{-1}(p) = \{ w \in S_n : P(w) = p\}.\]

\begin{thm}\label{thm:fiber}
For any Dyck path $p$, the following statements are true of $P^{-1}(p)$.
\begin{enumerate}\renewcommand{\labelenumi}{(\alph{enumi})}
\item There is a unique element $w' \in P^{-1}(p)$ of minimal length such that $\dep(w') = \ell_S(w')$.
\item For any other element $w'\neq w \in P^{-1}(p)$, $\dep(w) < \ell_S(w)$.
\end{enumerate}
\end{thm}

\begin{proof}
For part (a), we simply observe that if the elements that are not left-right maxima are arranged in increasing order, then every inversion of the permutation must have a left-right maximum as the larger number in the inversion pair. In other words, if we let $w'$ denote this permutation, we have
$$\{ (i,j) : i < j, w'(i) > w'(j) \} = \{ (i,j) : i<j, w'(i)>w'(j), \mbox{ and } (i,w'(i)) \in \lr(w')\}.$$
That is, the inequalities in Lemma \ref{obs:fiber} are equalities. For example the path in Figure \ref{fig:Dyck} has $w' = 23174856$, where $\{1,4,5,6\}$ have been inserted into $23\underline{\ \ }7\underline{\ \ }8\underline{\ \ }\,\underline{\ \ }$ in increasing order. See Figure \ref{fig:w'}.

\begin{figure}
\begin{center}
\begin{tikzpicture}[scale=.75, node distance=.5cm,>=stealth',bend angle=20,auto]
\draw[style=help lines,step=1cm] (0,0) grid (8,8);
\draw (.5,1.5) node[circle,draw] {2};
\draw (.5,.5) node {$\bullet$};
\draw (1.5,2.5) node[circle,draw] {3};
\draw (1.5,1.5) node {$\bullet$};
\draw (2.5,.5) node {1};
\draw (3.5,6.5) node[circle,draw] {7};
\draw (3.5,3.5) node {$\bullet$};
\draw (3.5,4.5) node {$\bullet$};
\draw (3.5,5.5) node {$\bullet$};
\draw (4.5,3.5) node {4};
\draw (5.5,7.5) node[circle,draw] {8};
\draw (5.5,5.5) node {$\bullet$};
\draw (5.5,6.5) node {$\bullet$};
\draw (6.5,4.5) node {5};
\draw (7.5,5.5) node {6};
\draw[very thick] (0,0) -- (0,2) -- (1,2) -- (1,3) -- (3,3) -- (3,7) -- (5,7) -- (5,8) -- (8,8);
\draw[dashed] (.5,.5) -- (2.3,.5);
\draw[dashed] (1.5,1.5) -- (2.3,.5);
\draw[dashed] (3.5,3.5) -- (4.3,3.5);
\draw[dashed] (3.5,4.5) -- (6.3,4.5);
\draw[dashed] (3.5,5.5) edge[bend right] (7.3,5.5);
\draw[dashed] (5.5,5.5) -- (6.3,4.5);
\draw[dashed] (5.5,6.5) -- (7.3,5.5);
\end{tikzpicture}
\end{center}
\caption{The minimal length element $w'$.}\label{fig:w'}
\end{figure}
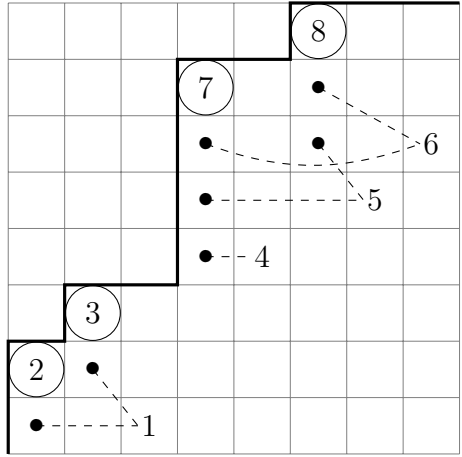

For part (b), there are two cases to consider for $w \in P^{-1}(p)$. 

If $w$ and $w'$ have the same excedance set, then $\dep(w) = \dep(w')$. If $w\neq w'$, the elements that are not left-right maxima cannot be increasing in $w$, and thus there must be an inversion among those numbers, giving $\ell_S(w) > \ell_S(w') =\dep(w') = \dep(w)$.

If $w$ has excedances apart from the left-right maxima, then we rely on Lemma \ref{obs:fiber} again, which says that if $w(i)-i = k >0$ and $(i,w(i))$ is \emph{not} a left-right maximum, then $w(i)$ is the greater value in more than $k$ inversions, and so
\begin{eqnarray*}
\dep(w) &=& \sum_{w(i) > i} (w(i)-i)\\
&<& \sum_{w(i) > i} |\{ j : i < j, w(i) > w(j)\}|\\
&\leq& \inv(w) = \ell_S(w),
\end{eqnarray*}
as desired.
\end{proof}

\begin{prop}
The minimal length fiber representatives, $w'$, are precisely the $321$-avoiding permutations.
\end{prop}

\begin{proof}
Let $w'$ be the minimal length representative of $P^{-1}(p)$. By definition, any left-right maximum can only play the role of `3' in the pattern 321. Thus, the other two elements in the pattern must be elements that are not left-right maxima. But by construction, the subword of these elements is strictly increasing, so it avoids the pattern 21.

On the other hand, it is straightforward to check that if $u$ and $v$ are two distinct permutations that avoid 321, then $\lr(u) \neq \lr(v)$, and so they correspond to distinct Dyck paths: $P(u)\neq P(v)$.
\end{proof}

We have now proved Theorem \ref{thm:hequalsl}.  The class of permutations defined in Theorem~\ref{thm:hequalsl} is entry P0002 of \cite{tenner dppa}, enumerated by sequence A000108 of \cite{oeis}.

\subsection{When reflection length, depth, and length coincide}\label{sec:lequalsl}

If we are given a permutation $w$ such that length equals reflection length, then Observation \ref{obs:1} clearly implies that the depth of $w$ has the same value. On the other hand, if we know only that depth equals reflection length, then Observation \ref{obs:2} claims that length equals reflection length and again all three are equal: \[\ell_T(w) = \dep(w) =\ell_S(w).\] It is this family of permutations that we study in this section.

The question of when $\ell_T(w) = \ell_S(w)$ has been studied before. Edelman \cite{edelman} studied the question of the joint distribution of the number of cycles, $n-\ell_T(w)$, together with the inversion number, $\ell_S(w)$. In particular, \cite[Theorem 3.1]{edelman} characterizes those $w$ for which $\ell_T(w) = \ell_S(w)$. These are what Edelman calls ``unimodal" permutations; that is, permutations in which the cycles are disjoint intervals and each cycle is unimodal. Here, we give a different characterization, in the language of pattern avoidance. 

\begin{thm}\label{thm:equivalent lengths}
The length and the reflection length of a permutation $w$ agree if and only if $w$ is $321$- and $3412$-avoiding.
\end{thm}

In the context of \cite{tenner patt-bru}, and the subsequent papers \cite{ragnarsson-tenner 1, ragnarsson-tenner 2} by Ragnarsson and the second author, such permutations were called \emph{boolean}.

The proof of Theorem \ref{thm:equivalent lengths} relies heavily on the main result of \cite{tenner rdpp}, where the second author provided a dictionary for translating between the one-line notation of a permutation and its reduced decompositions.  The second author extended this work in \cite{tenner patt-bru}.

\begin{lem}{\cite[Theorem 4.3]{tenner patt-bru}}\label{lem:boolean perm}
A permutation $w$ avoids both $321$ and $3412$ if and only if each (equivalently, any) of its reduced decompositions contains no repeated letters.
\end{lem}

Theorem~\ref{thm:equivalent lengths} will be proved using the previous lemma in conjunction with result of Dyer.

\begin{lem}{\cite[Theorem 1.1]{Dyer}}\label{lem:dyer erasing letters}
If $w = s_{i_1}\cdots s_{i_{\ell}}$ is a reduced decomposition for a permutation $w$ into simple reflections, then $\ell_T(w)$ is the minimum $p$ for which there exist indices $1 \le j_1 < \cdots < j_p \le \ell$ such that $e = s_{i_1} \cdots \widehat{s_{j_1}} \cdots \widehat{s_{j_p}} \cdots s_{\ell}$, where $\widehat{x}$ means to omit the symbol $x$.
\end{lem}

\begin{proof}[Proof of Theorem \ref{thm:equivalent lengths}]
Recall that the only relations among simple reflections are the Coxeter relations $(s_is_j)^{m(i,j)} = e$ where
$$m(i,j) = \begin{cases}
1 & i=j,\\
2 & |i-j| > 1,\text{ and}\\
3 & |i-j| = 1.
\end{cases}$$
Thus we know from Lemma~\ref{lem:dyer erasing letters} that $\ell_S(w) = \ell_T(w)$ if and only if a reduced decomposition for $w$ has no repeated letters.  This, by Lemma~\ref{lem:boolean perm} is equivalent to $w$ avoiding the patterns $321$ and $3412$.
\end{proof}

It is worth noting that we can also give an alternative proof to Theorem~\ref{thm:equivalent lengths}.  This alternative proof uses the following lemma.

\begin{lem}\label{lem:boolean means cycles are blocks}
If a permutation is both $321$- and $3412$-avoiding, then each cycle in its standard cycle notation consists of consecutive integers.
\end{lem}

\begin{proof}
The product of two cycles in which one consists of the integers $[a,b]$ and the other consists of the integers $[b,c]$ is a cycle consisting of the integers $[a,c]$.  By Lemma~\ref{lem:boolean perm}, a permutation avoiding $321$ and $3412$ has reduced decompositions with no repeated letters.  If each $s_k$ in such a reduced decomposition is written as the cycle $(k\ k+1)$, then we see that
\begin{itemize}
\item every cycle in the resulting product consists of consecutive integers, and
\item each letter appears in at most two cycles.
\end{itemize}
The first sentence of this proof guarantees that these two properties are maintained as we start to combine cycles by multiplication when they are not disjoint.  The process halts when no such multiplication is possible, meaning that no letter appears in more than one cycle.
\end{proof}

With Lemma~\ref{lem:boolean means cycles are blocks} in hand, the proof of Theorem~\ref{thm:equivalent lengths} only requires showing that the boolean Coxeter elements in $\mf{S}_n$ ($321$- and $3412$-avoiding permutations consisting of a single $n$-cycle) have length $n-1$.  This is easily done by induction on $n$.

The class of permutations defined in Theorem~\ref{thm:equivalent lengths} is entry P0006 of \cite{tenner dppa}, enumerated by sequence A001519 of \cite{oeis}, and enumerated by length in sequence A105306 of the same database.  These enumerations were calculated by the second author in \cite{tenner patt-bru}, in the context of boolean permutations.

\begin{cor}[\cite{tenner patt-bru}]\label{cor:enumeration}
Fix a positive integer $n$.  Then we have the following enumerative results, where $\{F_0, F_1, \ldots\}$ are the Fibonacci numbers.
\begin{align*}
|\{w \in \mf{S}_n : \ell_T(w) =\dep(w)= \ell_S(w)\}| &= F_{2n-1}\\
|\{w \in \mf{S}_n : \ell_T(w) = \dep(w)=\ell_S(w) =k\} &= \sum_{i=1}^{k} \binom{n-i}{k+1-i}\binom{k-1}{i-1}
\end{align*}
\end{cor}

As mentioned earlier, Theorem~\ref{thm:equivalent lengths} recovers a result of Edelman, although his result was phrased slightly differently.  It follows that the class of permutations defined in \cite{edelman} coincides with the set of permutations avoiding the patterns $321$ and $3412$.

\begin{cor}\label{cor:unimodal}
Unimodal permutations are exactly those permutations which avoid the patterns $321$ and $3412$.
\end{cor}

\begin{proof}
In \cite[Theorem 3.1]{edelman}, it was shown that a permutation $w$ satisfies $\ell_S(w) = \ell_T(w)$ if and only if $w$ is unimodal.  It follows from Theorem~\ref{thm:equivalent lengths}, then, that unimodal permutations are precisely those that avoid $321$ and $3412$.
\end{proof}


For a permutation $w \in \mf{S}_n$, let $\supp(w) \subseteq S_n$ be the set of letters appearing in reduced decompositions of $w$.  We call this set the \emph{support} of $w$.

Consider a permutation $w \in \mf{S}_n$.  We now have the following equivalent statements.
\begin{itemize}
\item $\ell_S(w) = \ell_T(w)$
\item all paths from $e$ to $w$ in the Bruhat graph have $\ell_S(w) = \ell_T(w)$ edges
\item $w$ avoids $321$ and $3412$
\item $|\supp(w)| = \ell_S(w) = \ell_S(w')$
\item the reduced decompositions of $w$ contain no repeated letters
\end{itemize}
We can now explore when, provided that $\ell_S(w) = \ell_T(w)$, all of the edges along all of these paths from $e$ to $w$ in the Bruhat graph are weighted by $1$. (And thus produce a path whose edge weights add up to depth.) That is, we can consider when all of these edges represent simple reflections $s \in S_n$, not elements of $T_n \setminus S_n$.

\begin{defn}
Suppose that a permutation $w$ has the property that if $s_k \in \supp(w)$ then $s_{k \pm 1} \not\in \supp(w)$.  Then we will say that $w$ is \emph{free}.
\end{defn} 

Note that all of the letters in a reduced decomposition of a free permutation $w$ are distinct, and they all commute with each other.  Thus $w$ has $\ell_S(w)!$ reduced decompositions.  The following lemma follows from the main result in \cite{tenner rdpp}.

\begin{lem}[See \cite{tenner rdpp}]\label{lem:free}
A permutation $w$ is free if and only if it avoids $231$, $312$, and $321$.
\end{lem}

\begin{cor}\label{cor:enumerating free}
The number of free permutations in $\mf{S}_n$ is the Fibonacci number $F_{n+1}$.
\end{cor}

\begin{proof}
The free permutations in $\mf{S}_n$ can be partitioned into two sets, determined by whether or not $s_{n-1}$ appears in their supports.  Those for which $s_{n-1}$ does not appear are in bijective correspondence with free permutations in $\mf{S}_{n-1}$.  Those $w$ for which $s_{n-1} \in \text{supp}(w)$ must have $s_{n-2}\not\in\text{supp}(w)$, meaning that they are in bijective correspondence with free permutations in $\mf{S}_{n-2}$.  Thus
$$|\{w \in \mf{S}_n : w\text{ is free}\}| = |\{w \in \mf{S}_{n-1} : w\text{ is free}\}| + |\{w \in \mf{S}_{n-2} : w\text{ is free}\}|.$$
Observing that $|\{w \in \mf{S}_1 : w\text{ is free}\}| = 1$ and $|\{w \in \mf{S}_2 : w\text{ is free}\}| = 2$ completes the proof.
\end{proof}

It is also possible to prove Corollary~\ref{cor:enumerating free} by noting, using Lemma~\ref{lem:free}, that the only allowable positions for $n$ in a free permutation $w$ are $w(n-1)$ and $w(n)$.

\begin{cor}\label{cor:free}
Suppose that $w\in\mf{S}_n$ satisfies the property $\ell_S(w) = \ell_T(w)$.  Then $w$ is free (that is, $w$ avoids $231$, $312$, and $321$) if and only if every decomposition $w = t_1 \cdots t_{\ell_S(w)}$ with $t_i \in T_n$ actually satisfies $t_i \in S_n$.
\end{cor} 

\begin{proof}
It follows from Theorem~\ref{thm:equivalent lengths} and Lemma~\ref{lem:boolean perm} that any reduced decomposition of $w$ consists of distinct letters.

Suppose that $w$ is not free.  Without loss of generality, we can assume that $w$ has a reduced decomposition of the form
$$\alpha\ s_k \ \beta\ s_{k+1} \ \gamma,$$
where $\alpha$, $\beta$, and $\gamma$ are products of simple reflections.  It is possible to use Coxeter relations to move the letters of $\beta$ in order to yield a reduced decomposition of $w$ having the form $\alpha' s_ks_{k+1}\gamma'$ : any letter $s_j \in \beta$ with $j<k$ moves past $s_{k+1}$ to the right, and any letter $s_j$ with $j>k+1$ moves past $s_k$ to the left.  But then we see the following equivalence
$$\alpha' \ s_ks_{k+1} \ \gamma' = \alpha' \ t_{k,k+2}s_k\ \gamma',$$
giving an undesirable decomposition of $w$ into $\ell$ reflections.

Now suppose that $w$ is free.  Then, for $k>1$ the only $k$-cycles in the standard cycle form of $w$ must be $2$-cycles of the form $(i,i+1)$.  In a decomposition of $w$ into reflections $t_1 t_2 \cdots t_{\ell}$, where $\ell = \ell_T(w)$, each successive permutation $w$, $wt_{\ell}$, $wt_{\ell}t_{\ell-1}, \ldots$ must have smaller length then the preceding permutation.  However, the only inversions in the permutation are of the form $\{i,i+1\}$, and so each $t_j$ must be a simple reflection.
\end{proof}

\begin{ex}
Consider the permutation $231 \in \mf{S}_3$.  By Lemma~\ref{lem:free}, we know that $231$ is not free.  It is easy to compute $\ell_S(231) = \ell_T(231) = 2$, and we see that
$$231 = s_1s_2 = t_{13}s_1.$$
\end{ex}

The class of permutations described in Corollary~\ref{cor:free} is entry P0026 of \cite{tenner dppa}, enumerated by sequence A000045 of \cite{oeis}. 

\section{Open questions and further remarks}\label{sec:open}

There are many possible directions for the future study of the depth statistic. For permutations, we have characterized depth combinatorially (Theorem \ref{thm:char}), and shown that another statistic, descent drop, has the same distribution.  The uniform upper bound for these statistics is $\lfloor n^2/4\rfloor$, and we know how many permutations achieve this bound (Proposition \ref{prp:upper2}). Section \ref{sec:depthk} discusses an expression for the generating function for depth derived in \cite{GP}. 

In Observation \ref{obs:1}, we showed depth was bounded below by the average of length and reflection length. While in Section \ref{sec:coincidence} we characterized those permutations for which depth equals length and those for which depth equals reflection length, it still remains to characterize the permutations for which depth equals the average of length and reflection length.

\begin{ques}\label{quesll'}
Which permutations $w$ have $\dep(w) = (\ell_T(w)+\ell_S(w))/2$?
\end{ques}

We have enumerated these permutations for small values of $n$ and found the sequence begins: \[ 1,2,6,23, 103, 511, 2719, 15205,\ldots.\] These terms match no sequence in \cite{oeis}.  Unlike the settings we studied in this work, the permutations whose depths achieve this lower bound are not characterized by pattern avoidance.  For example,
$$\dep(3412) = 4 > (2+4)/2 = (\ell_T(3412) + \ell_S(3412))/2,$$
while
$$\dep(53412) = 6 = (4+8)/2 = (\ell_T(53412) + \ell_S(53412))/2.$$
One approach to Question \ref{quesll'} would be to try to modify the approach of \cite{GP} to get a better understanding of the joint distribution of depth and length and/or reflection length.

Another open problem is to examine depth more closely in other Coxeter groups. It would be nice to find useful characterizations of depth analogous to the ``sum of sizes of excedances" characterization we have given for permutations. Surely such a model should exist in type $B_n$, at least. 

\begin{ques}
What is the analogue of Theorem \ref{thm:char} for signed permutations in type $B_n$?
\end{ques}

\begin{table}
\[ \begin{array}{c | c c c c c c c c c c c c c c c c}
B_n & k=0 & 1 & 2 & 3 & 4 & 5 & 6 & 7 & 8 & 9 & 10 & 11 & 12 & 13 & 14 & 15 \\
\hline
n=1 & 1 & 1 \\
2 & 1 & 2 & 4 & 1\\
3 & 1 & 3 & 8 & 13 & 14 & 8 & 1 \\
4 & 1 & 4 & 13 & 29 & 55 & 66 & 90 & 53 & 60 & 12 & 1 \\
5 & 1 & 5 & 19 & 52 & 120 & 219 & 340 & 457 & 594 & 556 & 505 & 466 & 325 & 164 & 16 & 1
\end{array} \]
\bigskip
\caption{Depth distribution for $B_n$, that is, $|\{w \in B_n: \dep(w)=k\}|$.}\label{tab:BC}
\end{table}

We have the distributions of depth in $B_n$ for $n\leq 5$ in Table \ref{tab:BC}. It certainly appears that $\dep(w)$ is bounded above by $\binom{n+1}{2}$, although we have no proof of this. Also interesting is the number of elements of $B_n$ for which depth equals length. For $n =1,2,\ldots, 6$ we get:
\[ 2, 5, 14, 42, 132, 429 \]
so it is reasonable to guess that $|\{ w\in B_n : \dep(w) = \ell_S(w) \}|=\Cat_{n+1}$.  Stembridge showed in \cite[Proposition 5.9d)]{stem} that $\Cat_{n+1}$ counts the number of ``fully commutative top-and-bottom elements" of $B_n$. Perhaps this characterizes $\dep(w) = \ell_S(w)$ in type $B_n$, just as $\dep(w) = \ell_S(w)$ in $\mf{S}_n$ if and only if $w$ is fully commutative.

\bigskip

\noindent\textbf{Acknowledgements} We would like to thank Drew Armstrong, Anders Claesson, Ira Gessel, Einar Steingr\'imsson, John Stembridge, and Alex Woo for comments on an early draft of this paper. Thanks also to Timothy Walsh for his thoughts on the distribution of depth for permutations, and to the thoughtful suggestions of the anonymous referees.

\end{document}